\documentclass{article}

\usepackage{algorithm,algorithmic}
\usepackage{amsfonts}
\usepackage{amsmath}
\usepackage{amssymb}
\usepackage{amsthm}
\usepackage{bm}
\usepackage[title]{appendix}
\usepackage{booktabs}
\usepackage{enumitem}
\usepackage[margin=1.25in]{geometry}
\usepackage{graphicx}
\usepackage{hyperref}
\usepackage{lineno}
\usepackage{mathrsfs}
\usepackage{mathtools}
\usepackage{multirow}
\usepackage{multicol}
\usepackage{subfig}
\usepackage{graphicx} 
\usepackage{ulem}
\usepackage{xcolor}
\allowdisplaybreaks[4]

\theoremstyle{plain}
\newtheorem{The}{Theorem}[section]  
\newtheorem{lem}{Lemma}[section]
\newtheorem{cor}{Corollary}[section]
\newtheorem{prop}{Proposition}[section]

\theoremstyle{definition}
\newtheorem{Def}{Definition}[section] %

\theoremstyle{remark}
\newtheorem{remark}{Remark}[section]

\newcommand{\abs}[1]{\left|#1\right|}

\hypersetup{colorlinks,breaklinks}
\hypersetup{
	colorlinks=true,
	allcolors={green!50!black},
	urlcolor={red!50!black}
}

\newcommand{\eps}{\varepsilon}

\newcommand{\supp}{\mathrm{supp}}

\newcommand{\calA}{\mathcal{A}}
\newcommand{\calB}{\mathcal{B}}

\newcommand{\calF}{\mathcal{F}}

\newcommand{\calI}{\mathcal{I}}

\newcommand{\calL}{\mathcal{L}}

\newcommand{\calO}{\mathcal{O}}

\newcommand{\calS}{\mathcal{S}}

\newcommand{\C}{\mathbb{C}}
\newcommand{\N}{\mathbb{N}}
\newcommand{\R}{\mathbb{R}}

\newcommand{\bfun}{\varphi}
\newcommand{\sgn}{\text{sgn}}
\newcommand{\cp}{\mathrm{cp}}
\newcommand{\productspace}[2]{\bigotimes^{#1}\mathbb{C}^{#2}}

\newcommand{\ba}{\boldsymbol{a}}
\newcommand{\bb}{\boldsymbol{b}}

\newcommand{\rr}{\boldsymbol{r}}

\newcommand{\bell}{\boldsymbol{\ell}}
\newcommand{\bx}{\boldsymbol{x}}
\newcommand{\bX}{\mathbf{X}}
\newcommand{\bY}{\mathbf{Y}}
\newcommand{\bE}{\mathbf{E}}
\newcommand{\bB}{\mathbf{B}}
\newcommand{\bu}{\boldsymbol{u}}

\newcommand{\bk}{\boldsymbol{k}}
\newcommand{\be}{\boldsymbol{e}}
\newcommand{\btheta}{\boldsymbol{\theta}}

\newcommand{\scrH}{\mathscr{H}}

\newcommand{\scrT}{\mathscr{T}}

\newcommand{\scrA}{\mathscr{A}}

\newcommand{\lrbrace}[1]{\left\{#1\right\}}
\newcommand{\lrbracket}[1]{\left(#1\right)}
\newcommand{\lrsquare}[1]{\left[#1\right]}

\newcommand{\inner}[1]{\left\langle#1\right\rangle}
\newcommand{\norm}[1]{\left\Vert#1\right\Vert}

\newcommand{\rank}{\mathrm{rank}}

\definecolor{lightblue}{rgb}{0.957,0.963,0.975}
\definecolor{midblue}{rgb}{0.937,0.943,0.965}
\definecolor{deepblue}{rgb}{0.325,0.427,0.569}
\definecolor{blocktitleblue}{rgb}{0.225,0.427,0.669}
\definecolor{lightred}{rgb}{0.996,0.969,0.969}
\definecolor{midred}{rgb}{0.976,0.949,0.949}
\definecolor{deepred}{rgb}{0.686,0.133,0.098}
\definecolor{deepgreen}{rgb}{0,0.5,0}
\definecolor{halfgray}{gray}{0.55}

\newcommand{\blue}[1]{\textcolor{blue}{#1}}

\definecolor{lightpurple}{rgb}{0.978,0.978,1.0}
\definecolor{deeppurple}{rgb}{0.353,0.275,0.478}

\begin{document}

\title{Lower Bound on the Representation Complexity of Antisymmetric Tensor Product Functions}
\date{\today}

\author{Yuyang Wang\thanks{State Key Laboratory of Mathematical Sciences
		, Academy of Mathematics and Systems Science, Chinese Academy of Sciences, Beijing 100190, China and School of Mathematical Sciences, University of Chinese Academy of Sciences, Beijing 100049, China (email: \href{mailto:wangyuyang@lsec.cc.ac.cn}{\blue{wangyuyang@lsec.cc.ac.cn}}).} \and 
	Yukuan Hu\thanks{CERMICS, \'Ecole des Ponts, IP Paris, 77455 Marne-la-Vall\'ee, France (email: \href{mailto:ykhu@lsec.cc.ac.cn}{\blue{ykhu@lsec.cc.ac.cn}}).} \and Xin Liu\thanks{State Key Laboratory of Mathematical Sciences, Academy of Mathematics and Systems Science, Chinese Academy of Sciences, Beijing 100190, China and School of Mathematical Sciences, University of Chinese Academy of Sciences, Beijing 100049, China (e-mail: \href{mailto:liuxin@lsec.cc.ac.cn}{\blue{liuxin@lsec.cc.ac.cn}}).}}

\maketitle

\abstract{Tensor product function (TPF) approximations have been widely adopted in solving high-dimensional problems, such as partial differential equations and eigenvalue problems, achieving desirable accuracy with computational overhead that scales linearly with problem dimensions. However, recent studies have underscored the extraordinarily high computational cost of TPFs on quantum many-body problems, even for systems with as few as three particles. A key distinction in these problems is the antisymmetry requirement on the unknown functions. In the present work, we rigorously establish that the minimum number of involved terms for a class of TPFs to be exactly antisymmetric increases exponentially fast with the problem dimension. This class encompasses both traditionally discretized TPFs and the recent ones parameterized by neural networks. Our proof exploits the link between the antisymmetric TPFs in this class and the corresponding antisymmetric tensors and focuses on the Canonical Polyadic rank of the latter. As a result, our findings reveal that low-rank TPFs are fundamentally unsuitable for high-dimensional problems where antisymmetry is essential.}


\section{Introduction} 

\par  High-dimensional partial differential equations (PDEs) and eigenvalue problems frequently arise in scientific and engineering applications. In solving these problems, traditional numerical methods, such as finite difference and finite element methods, suffer from the curse of dimensionality, in that both the storage and computational costs grow exponentially with the problem dimension. For example, discretizing an $N$-dimensional domain with merely two grid points per dimension can result in a dense tensor of size $2^N$, requiring over 1 ZB of memory for direct storage in double precision when $N\ge 70$. 

\par To circumvent the curse of dimensionality, tensor product functions (TPFs) have been proposed to approximate the original high-dimensional ones \cite{bachmayr2023low,beylkin2002numerical,hackbusch2012tensor}. For a function $f$ defined on a Cartesian product domain $\bigtimes_{j=1}^N\Omega_j$, with $\Omega_j\subseteq \R^d$ ($j=1,\ldots,N$), its TPF approximation is of the form
\begin{equation*}
	f(\rr_1,\ldots,\rr_N) \approx \sum_{i=1}^p\psi_{i1}(\rr_1)\cdots\psi_{iN}(\rr_N),\quad\forall~\rr=(\rr_1,\ldots,\rr_N)\in\bigtimes_{j=1}^N\Omega_j,
\end{equation*}
where $p\in\N$ is called the separation rank in the literature \cite{beylkin2002numerical,beylkin2005algorithms} (see Remark \ref{rem:TPF and separation rank} for a rigorous definition) and $\psi_{ij}$ is a function defined on $\Omega_j$ ($i=1,\ldots,p$, $j=1,\ldots,N$). 
Combined with the traditional discretization methods for $\{\psi_{ij}\}$, the TPF approximation requires a computational cost growing linearly with respect to $N$. The foundation of TPFs can be traced back to the work of Schmidt \cite{schmidt1908theorie} on the case of $N=2$, $d=1$ with $\psi_{ij} \in \calL^2(\Omega_j)$, known as the Schmidt decomposition. Up to this point, this construction has been developed to handle higher-dimensional settings and function spaces with more favorable properties \cite{bazarkhanov2015nonlinear,beylkin2002numerical,beylkin2005algorithms,dolgov2021functional,griebel2023analysis,griebel2023low,hackbusch2007tensor,hackbusch2008tensor,mohlenkamp2005trigonometric,oseledets2013constructive}.

\par In recent years, neural networks have gained significant attention for their powerful approximation capabilities (see \cite{devore2021neural} and the references therein for a comprehensive review) and have provided mesh-free solutions for PDEs (see, e.g., \cite{dai2024subspacemethodbasedneural,weinan2018deep,huang2024adaptiveneuralnetworkbasis,Yulei_Liao_2021,lin2024adaptiveneuralnetworksubspace,raissi2019physics,siegel2023greedy,sirignano2018dgm,xiao2024learningsolutionoperatorspdes,yu2024naturaldeepritzmethod}). Particularly in the high-dimensional regime, Wang et al. \cite{wang2022tensor} introduce the tensor neural networks (TNNs) by utilizing neural networks to construct $\psi_{ij}$ in the TPF approximation, effectively eliminating the need for high-dimensional integrals (see Section \ref{subsec:TPF and TNN} for an illustration). TNN approximations have demonstrated high accuracy across various problems \cite{Hu_2024,kao2024petnns,liao2022solving,wang2024tensorFP,wang2022tensor,wang2024solving,wang2024computing}; for instance, with a separation rank of $p\le30$, they achieve errors of $\calO(10^{-7})$ for a 20,000-dimensional Schr\"odinger equation with coupled quantum harmonic oscillator potential \cite{Hu_2024}.

\par However, recent studies have highlighted the limitations of TNN approximations in addressing the electronic Schr\"odinger equation in quantum mechanics \cite{liao2024solvingschrodingerequationusing}. This problem differs from the aforementioned applications of TNN approximations in the additional antisymmetry constraint on the unknown functions. 
Mathematically, an antisymmetric function $f$ defined on $\Omega^N$ ($\Omega\subseteq\R^d$) satisfies
	\begin{equation}
		\label{eq:antisym}
		f\lrbracket{\scrT_{ij}(\rr_1,\dots,\rr_N)} = -f(\rr_1,\dots,\rr_N),\quad \forall~i\neq j,
	\end{equation}
	where $\scrT_{ij}:\Omega^N \mapsto \Omega^N$ swaps the $i$-th and $j$-th coordinates while keeping all the others unchanged.
In the electronic Schr\"odinger equation, 
the electronic wave function describes the behavior of electrons, whose antisymmetry is grounded in the fundamental Pauli exclusion principle for fermions \cite{pauli1925zusammenhang}. In \cite{liao2024solvingschrodingerequationusing}, the authors approximate the electronic wave function with TNNs and optimize the neural network parameters through a variational principle. To achieve the desired accuracy even on systems with as few as three electrons, they have to employ TNNs with $p=50$ and undergo extraordinarily lengthy optimization processes.

\par The challenge can stem from the representation of antisymmetry. We first note that TPFs are not inherently antisymmetric. For instance, if we set $\psi_{i1}=\psi_{i2}=\cdots=\psi_{iN}$ for each $i$, the resulting function is symmetric. However, antisymmetry can be crucial for applications of interest. In quantum mechanics, the antisymmetry of wave function ensures that identical fermions cannot occupy the same quantum state. Adopting antisymmetric ans\"atze thus guarantees the physical meaningfulness of results, as already observed in 1930 \cite{fock1930naherungsmethode,slater1930note}. We also perform numerical comparisons between TPFs and their antisymmetrization. Figure \ref{fig:tnn_comparison} reveals that applying antisymmetrization improves both accuracy and optimization efficiency. Detailed experimental settings are provided in Appendix \ref{sec_appendix}. Consequently, existing works on representing antisymmetric wave functions usually rely on determinant-based constructions, such as Slater-type wave functions \cite{feynman1956energy,fock1930naherungsmethode,hutter2020representing,jastrow1955many,pfau2020ab,slater1930note,zhou2024multilevel}, pairwise wave functions \cite{han2019universal,han2019solving,pang2022n}, Pfaffian wave functions \cite{bajdich2006pfaffian,bajdich2008pfaffian,gao2024neural}, explicit antisymmetrization \cite{lin2023explicitly}, and some implicit constructions \cite{chen2023exact,ye2024widetilde}. Within the framework of TPF approximation, these constructions correspond to separation ranks on the order of $N!$, rendering the computational advantages from the separability of TPFs inapplicable. Together, these findings motivate us to investigate the representation complexity of TPFs for antisymmetry, which refers to the minimum number of TPF terms required to represent antisymmetric functions. 

\begin{figure}[!t]
	\centering
	\begin{minipage}[b]{0.49\textwidth}
		\centering
		\includegraphics[width=\textwidth]{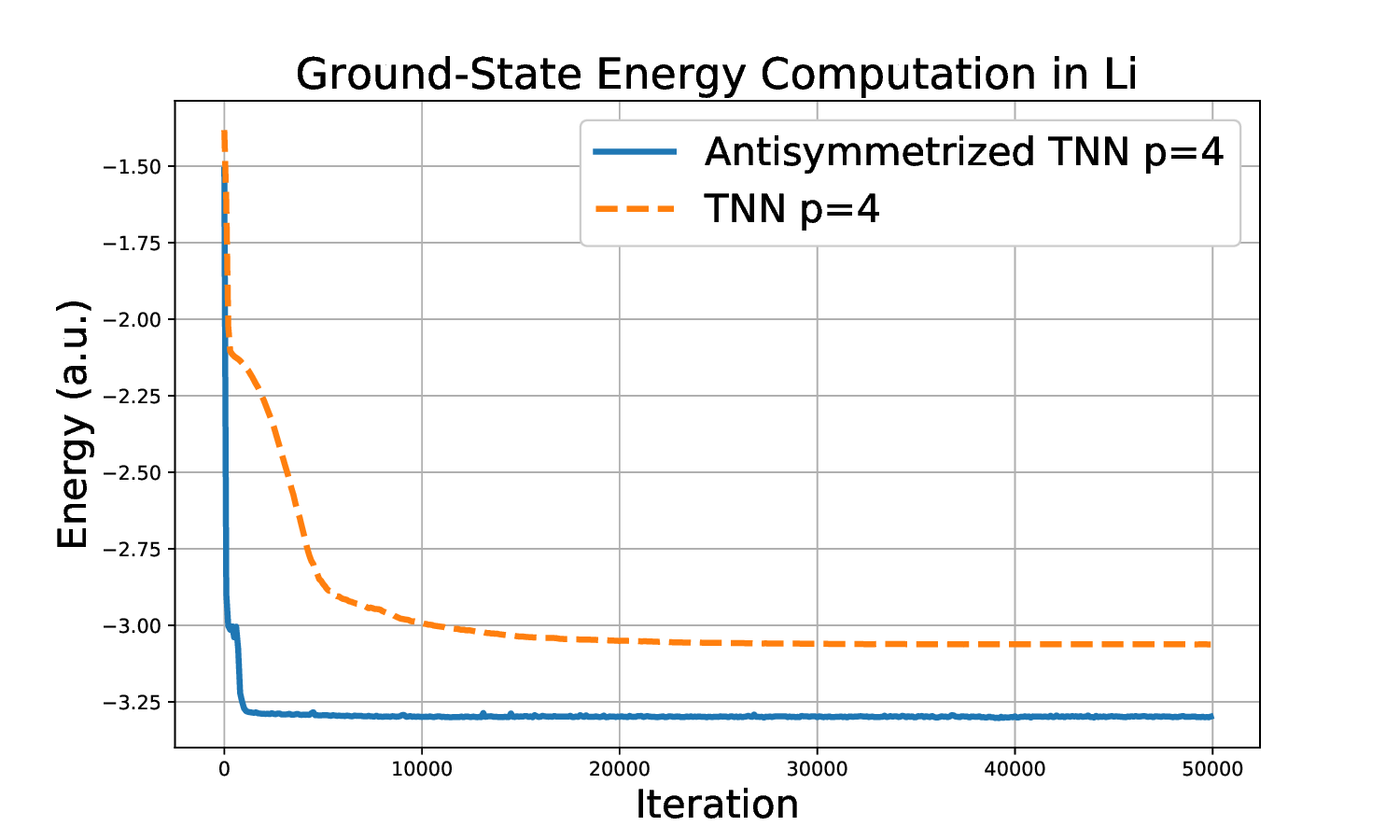}
	\end{minipage}
	\hfill
	\begin{minipage}[b]{0.49\textwidth}
		\centering
		\includegraphics[width=\textwidth]{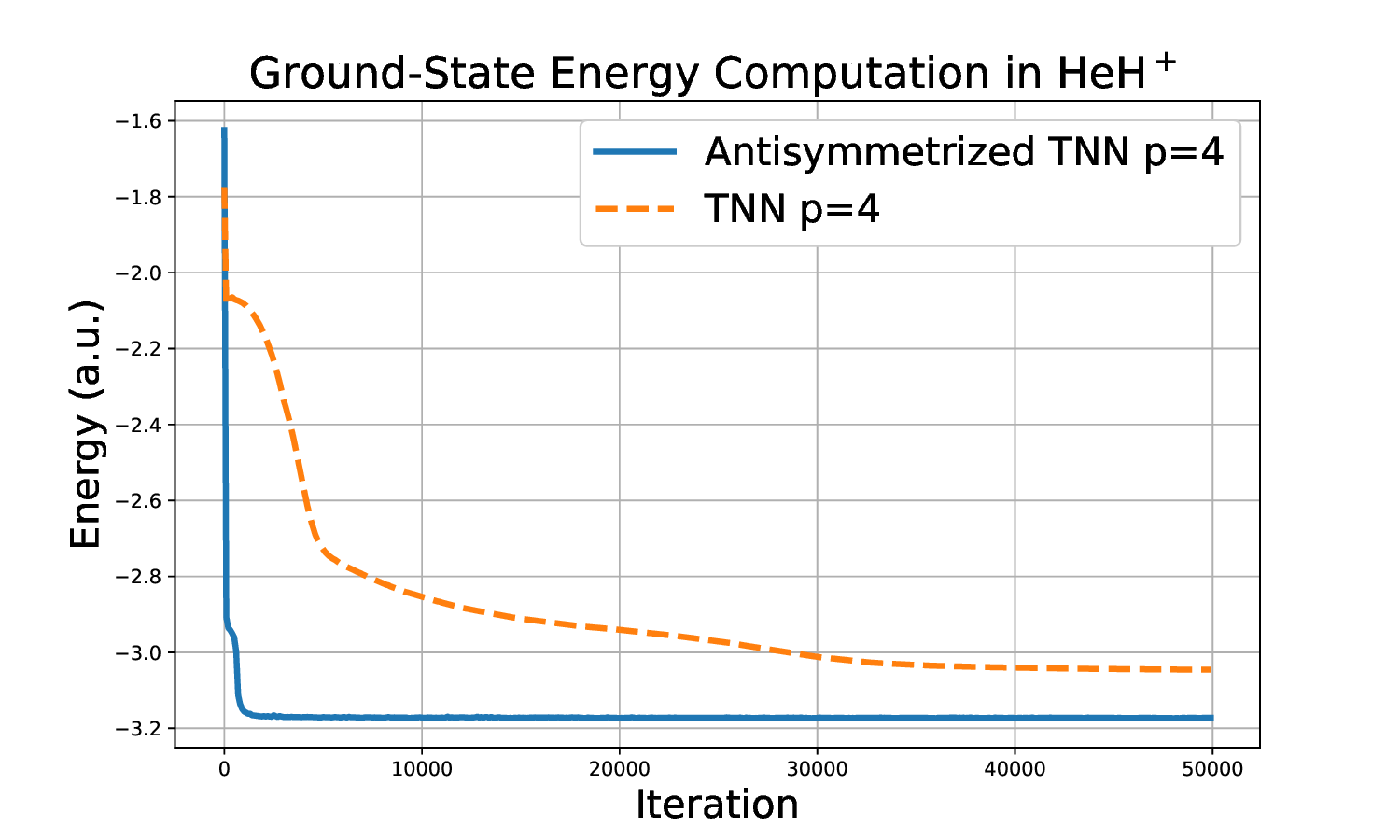}
	\end{minipage}
	\caption{Numerical comparison of the TNN approximations with and without explicit antisymmetrization on the one-dimensional systems of $\mathrm{Li}$ (left) and \(\mathrm{HeH}^+\) (right).}
	\label{fig:tnn_comparison}
\end{figure}

\vskip 0.2cm

\par \textbf{Contributions.} 
In this work, we examine a class of TPFs in which the functions $\{\psi_{ij}\}$ reside in finite-dimensional spaces. This class covers both TPFs discretized using traditional methods and the TNNs with any fixed network architecture. By linking the antisymmetric TPFs in this class to the associated antisymmetric tensors and analyzing the canonical polyadic rank of the latter, we establish an exponential lower bound of order $2^N/\sqrt{N}$ for the minimum number of terms required in the antisymmetric TPF representations. Our findings reveal that low-rank TPFs are fundamentally unsuitable for high-dimensional problems with the antisymmetry constraint, including the electronic Schr\"odinger equation in quantum mechanics.

\vskip 0.2cm

\par \textbf{Organization.} This paper is organized as follows: Section~\ref{sec_pre} introduces fundamental concepts related to TPFs, TNNs, and antisymmetric functions and tensors. Section~\ref{sec_result} presents our main theoretical results on the antisymmetric TPFs and their application to the antisymmetric TNNs with fixed network architectures. Finally, Section~\ref{sec_discussion} discusses the implications of these findings for relevant applications.

\vskip 0.2cm

\par \textbf{Notations.} Throughout this paper, scalars, vectors, and tensors are denoted by lowercase letters (e.g., $x$), bold lowercase letters (e.g., $\bx$), and uppercase bold letters (e.g., $\bX$), respectively. Operators are represented by uppercase script letters (e.g., $\scrA$). We use subscripts to indicate components or blocks of vectors and matrices. For example, $x_i$ denotes the $i$-th component of $\bx$, and $\rr_i$ denotes the $i$-th block of $\rr$. For an $N$-order tensor $\bX$, the entry indexed by $(k_1,\dots,k_N)$ is written as $\bX(k_1,\dots,k_N)$. The support of $\bX$ is denoted by $\supp(\bX):=\{(k_1,\ldots,k_N):\bX(k_1,\ldots,k_N)\ne0\}$.

\par We use $\calB(\Omega)$ for the Banach space of functions defined on the domain $\Omega$. The Cartesian products of sets are denoted by either exponents (e.g., $\Omega^N$) or multiplication symbols (e.g., $\Omega_1\times\Omega_2$ and $\bigtimes_{j=1}^N\Omega_j$). We use the notation $\otimes$ or $\bigotimes$ to represent tensor products of vectors (e.g., $\bx_1\otimes\bx_2$ and $\bigotimes_{j=1}^N\bx_j$), vector spaces (e.g., $\C^{K_1}\otimes\C^{K_2}$ and $\bigotimes_{j=1}^N\C^{K_j}$), functions (e.g., $\psi_1\otimes\psi_2$ and $\bigotimes_{j=1}^N\psi_j$), and function spaces (e.g., $\calB(\Omega_1)\otimes\calB(\Omega_2)$ and $\bigotimes_{j=1}^N\calB(\Omega_j)$). Note that the tensor product of two functions is defined as $(\psi_1\otimes\psi_2)(\rr_1,\rr_2):=\psi_1(\rr_1)\psi_2(\rr_2)$ for any $\rr_1$ and $\rr_2$. The tensor product of function spaces can be obtained in an analogous way (see Definition \ref{def:tpf}). The notation $\binom{K}{N}$ represents the binomial coefficient, defined as $\binom{K}{N}=K!/(N!(K-N)!)$. We denote the permutation group over $\{1,\ldots,N\}$ as $\calS_N$. The notation $\mathrm{sgn}(\pi)$ gives the sign of the permutation $\pi$. We use $\Theta(\cdot)$ to denote quantities that are of the same asymptotic order.


\section{Preliminaries}
\label{sec_pre}

\par This section presents the concepts and tools necessary for our main results. We begin by detailing the structures of TPFs including TNNs, which provide the core frameworks for approximating high-dimensional functions. Following this, we introduce the mathematical properties of antisymmetric functions and tensors. For the broadest generality and compatibility with real applications, all functions in this work are assumed to be complex-valued and defined on real spaces.

\subsection{Tensor Product Function and Tensor Neural Network}\label{subsec:TPF and TNN}

\par We begin by defining TPFs over general Banach spaces.
\begin{Def}[TPF]\label{def:tpf}
	Let $\calB(\Omega_j)$ be a Banach space of functions defined on $\Omega_j\subseteq \R^d$ ($j=1,\ldots,N$). The space of TPFs on $\{\calB(\Omega_j)\}_{j=1}^N$ is defined as
	\begin{equation}
		\label{eq:tpf}
		\bigotimes_{j=1}^N \calB(\Omega_j) := \lrbrace{f\in \calB(\bigtimes_{j=1}^N\Omega_j): f=\sum_{i=1}^p \bigotimes_{j=1}^N\psi_{ij},~\text{where}~p\in\N,~\psi_{ij}\in \calB(\Omega_j),~\forall~i,j},
	\end{equation}
	and each function $f\in\bigotimes_{j=1}^N \calB(\Omega_j)$ is called a TPF (on $\{\calB(\Omega_j)\}_{j=1}^N$).
\end{Def}

\begin{remark}
	The notation $\bigotimes_{j=1}^N \calB(\Omega_j)$ sometimes stands for the closure of the right-hand side of Eq. \eqref{eq:tpf} \cite[Section 4.2.1]{hackbusch2012tensor}, ensuring completeness. Since we focus on the functions with explicit tensor product structures, we omit the closure and, with slight abuse of notation, still denote the resulting space by $\bigotimes_{j=1}^N\calB(\Omega_j)$.
\end{remark}

\par Unless stated, we assume that $\Omega_1=\cdots=\Omega_N=\Omega$, and we write $\bigotimes^N \calB(\Omega):=\bigotimes_{j=1}^N \calB(\Omega_j)$. Since for any separable Hilbert space $\calB(\Omega)$, $\bigotimes^N \calB(\Omega)$ is dense in $\calB(\Omega^N)$ for any compact $\Omega\subseteq \R^{d}$ \cite[Theorem 3.12]{weidmann1980operators}, TPFs can in principle approximate any function in $\calB(\Omega^N)$ to arbitrary accuracy. 

\par To facilitate our analysis, we define the rank of a TPF.

\begin{Def}[TPF rank]
	\label{def:tpf_rank}
	The TPF rank of a function $f\in \bigotimes^N \calB(\Omega)$, denoted by $\rank_{\calB(\Omega)}(f)$, is defined as
	$$\rank_{\calB(\Omega)}(f) := \min \lrbrace{p\in\N: f=\sum_{i=1}^p\bigotimes_{j=1}^N\psi_{ij},~\text{where}~\psi_{ij} \in \calB(\Omega),~\forall~i,j},$$
	or in other words, the minimum number of terms for the TPF representation of $f$.
\end{Def}

\begin{remark}\label{rem:TPF and separation rank}
	The definitions of TPF rank and separation rank \cite{bazarkhanov2015nonlinear,beylkin2002numerical,beylkin2005algorithms,griebel2023analysis} differ fundamentally. The separation rank of a function is typically used in the context of approximation \cite{bazarkhanov2015nonlinear,beylkin2002numerical,beylkin2005algorithms,griebel2023analysis}: Given a tolerance $\eps>0$, the separation rank of a function $f$ is any $p\in\N$ such that there exist $\psi_{ij}\in\calB(\Omega)$ ($i=1,\ldots,p$, $j=1,\ldots,N$) satisfying
	$$\norm{f-\sum_{i=1}^p\bigotimes_{j=1}^N\psi_{ij}}_{\calB(\Omega^N)}\le\eps$$
	The optimal (or minimal) separation rank of $f$ then refers to its minimum separation rank for the given tolerance. 
	
	\par In contrast, the TPF rank is defined specifically for TPFs, capturing the minimal number of terms required for an exact TPF representation. We shall emphasize that the exact representation of antisymmetry is meaningful in applications of interest. For example, the antisymmetry of wave function in quantum mechanics guarantees the physical correctness of results \cite{fock1930naherungsmethode,slater1930note}. Intuitively, it is also natural to approximate antisymmetric functions with candidates in the same class. 
\end{remark}

\par The TNN can be viewed as a special case of TPF with $d=1$ and $\{\psi_{ij}\}$ parameterized by neural networks \cite{wang2022tensor}. Each input $r_j\in\Omega\subseteq\R$ is independently processed through a fully connected subnetwork that outputs a $p$-dimensional vector:
\begin{equation*}
	r_j \xlongrightarrow[\text{subnetwork}]{\text{the $j$-th fully connected}} \lrbracket{\psi_{1j}(r_j;\btheta_j), \dots, \psi_{pj}(r_j;\btheta_j)}^{\top},
\end{equation*}
where $\btheta_j$ represents the parameters of the $j$-th subnetwork. Let $\btheta$ collect the parameters from $N$ subnetworks. The TNN function $f_{\btheta}$ is then defined as
\begin{equation}
	\label{eq:TNN_general_fun}
	f_{\btheta}:= \sum_{i=1}^p\bigotimes_{j=1}^N \psi_{ij}(\cdot;\btheta_{j}),
\end{equation}
Figure~\ref{fig:TNN} illustrates the architecture of TNN. 
\begin{figure}[htb]
	\centering
	\includegraphics[width=0.8\linewidth]{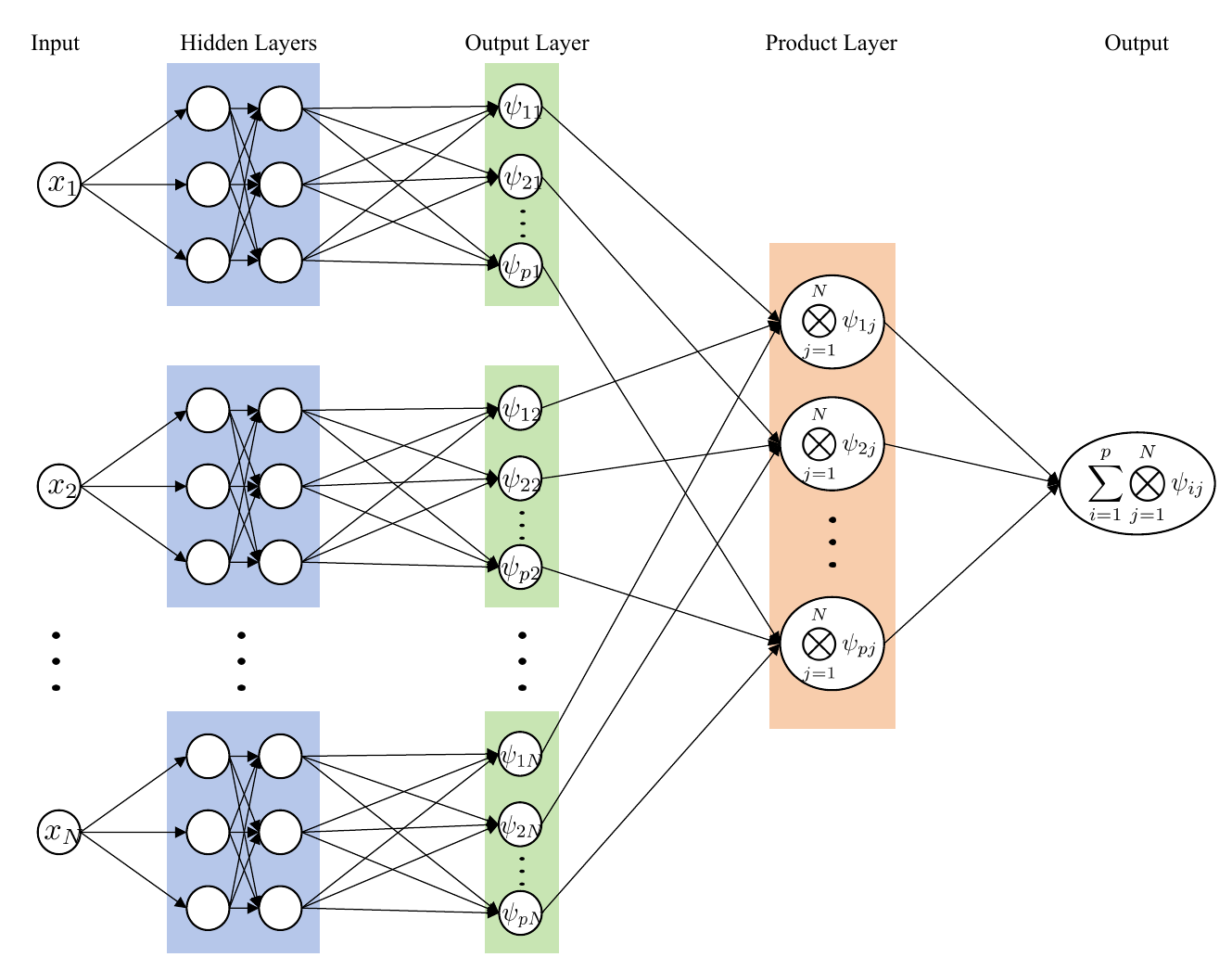}
	\caption{Architecture of TNN.}
	\label{fig:TNN}
\end{figure}

\par In the context of quantum mechanics, to approximate the electronic wave function of an isolated, non-relativistic, time-independent, and spinless $N$-electron system in $\Omega\subseteq\R^d$, the TNN wave function ansatz follows a similar structure. Each electronic coordinate is represented by a $d$-dimensional vector $\rr_j=(r_{j1},\ldots,r_{jd})^{\top} \in \R^d$, with each one-dimensional coordinate processed by an independent subnetwork, resulting in a total of $dN$ independent subnetworks \cite{liao2024solvingschrodingerequationusing}. The TNN wave function ansatz for this system is
\begin{equation}
	\label{eq:tnn_wf} 
	f_{\btheta}(\rr) := \sum_{i=1}^p\prod_{j=1}^N\lrbracket{\prod_{k=1}^d\psi_{ijk}(r_{jk};\btheta_{jk})},\quad\forall~\rr\in\Omega^N,
\end{equation} 
where $\btheta$ collects the parameters of $dN$ subnetworks.

\subsection{Antisymmetric Functions}\label{subsec:antisym fun}

\par Given a domain $\Omega\subseteq \R^d$, the antisymmetric function $f$ defined on $\Omega^N$ changes sign when any two inputs are exchanged, as defined in Eq. \eqref{eq:antisym}. Equivalently,
\begin{equation}
	\label{eq:antisym_perm}
	f(\rr_1,\dots,\rr_N) = \sgn(\pi)f(\rr_{\pi(1)},\dots,\rr_{\pi(N)}), \ \forall~\pi \in \calS_N,~\rr\in \Omega^N,
\end{equation}
where $\calS_N$ is the permutation group on $\{1,\dots,N\}$, and $\sgn(\pi)$ is the sign of permutation $\pi$. The set of all antisymmetric functions in $\bigotimes^N\calB(\Omega)$ is denoted by 
$$\calA\lrbracket{\bigotimes^N\calB(\Omega)} := \lrbrace{f\in \bigotimes^N\calB(\Omega): f\text{ satisfies Eq. \eqref{eq:antisym_perm}}}.$$

\par The antisymmetrizer (also known as the antisymmetrizing operator) \cite{lowdin1958correlation} on $\Omega^N$, denoted by $\scrA$, is a linear operator that antisymmetrizes functions. For any $f$ defined on $\Omega^N$, the action of $\scrA$ on $f$ is defined as
\begin{equation}
	\label{eq:anti_proj}
	\scrA[f](\rr_1,\dots,\rr_N) := \dfrac{1}{N!}\sum_{\pi \in \calS_{N}}\sgn(\pi)f(\rr_{\pi(1)},\dots,\rr_{\pi(N)}),\quad \forall~\rr\in\Omega^N.
\end{equation}
For a function $f\in\bigotimes^N\calB(\Omega)$, it belongs to $\calA(\bigotimes^N\calB(\Omega))$ if and only if $f = \scrA[f]$. When $\bigotimes^N\calB(\Omega)$ is a Hilbert space, the antisymmetrizer $\scrA$ acts as the orthogonal projection operator onto $\calA(\bigotimes^N\calB(\Omega))$.
\begin{remark}
	\label{re:hf_slater}
	Let us revisit the TPF in Eq. \eqref{eq:tpf}. When $p=1$, it reduces to the Hartree product \cite{hartree1928wave}, an early wave function ansatz in the quantum mechanics literature that lacks antisymmetry. When $p=N!$ with $\psi_{i1} = \text{sign}(\pi_i) \psi_{\pi_i(1)}$ and $\psi_{ij} = \psi_{\pi_i(j)}$ for $1<j\le N$, where $\pi_i \in \calS_N$ for $i=1,\ldots,N!$, the TPF becomes the Slater determinant \cite{fock1930naherungsmethode,slater1930note}, a well-known wave function ansatz or ingredient in various Hartree-Fock-based approximations for electronic structure calculations \cite[Chapter 5]{helgaker2000molecular}. If antisymmetrized as defined in Eq. \eqref{eq:anti_proj}, the TPF becomes a sum of $p$ Slater determinants, whose expansion contains at most $pN!$ terms, and thus the TPF rank should satisfy $1 <\rank_{\calB(\Omega)}(\scrA[f])\le pN!$. Note that this seemingly daunting exponential rank does not necessarily imply an unacceptably high computational cost; the Slater determinant can be implemented in polynomial time due to its algebraic structure.
\end{remark}

\subsection{Antisymmetric Tensors}

\par The space of $N$-order $(K_1,\dots,K_N)$-dimensional tensors, denoted by $\bigotimes_{j=1}^N\C^{K_j}$, is defined as
$$\bigotimes_{j=1}^N\C^{K_j} := \C^{K_1}\otimes \cdots \otimes \C^{K_N}.$$
In the special case where $K_1=\cdots=K_N=K$, we denote this space as $\productspace{N}{K}$. Since the TPF rank is highly relevant to the Canonical Polyadic (CP) rank of tensor (as will be shown later), we provide the formal definition of the latter as follows. 
\begin{Def}[CP Rank \cite{hitchcock1927expression}]
	The CP rank of a tensor $\bX \in \bigotimes_{j=1}^N\C^{K_j}$, denoted by $\rank_{\cp} (\bX)$, is defined as
	$$\rank_{\cp} (\bX) := \min \lrbrace{p\in \N: \bX= \sum_{i=1}^p \bx_{i1}\otimes \cdots \otimes \bx_{iN}, \ \bx_{ij}\in \C^{K_j},~1\le i\le p,~1\le j \le N}.$$
\end{Def}

\par In analogy to the antisymmetric function, an antisymmetric tensor changes sign with the exchange of any two indices. Following this, we denote the closed subspace of $\bigotimes^N\C^K$ for all antisymmetric tensors as
$$\calA\lrbracket{\productspace{N}{K}} := \lrbrace{\bX \in \productspace{N}{K}: \bX(k_1,\ldots,k_N)=\sgn(\pi)\bX(k_{\pi(1)},\dots,k_{\pi(N)}),\forall~\pi \in \calS_N}.$$
The unique orthogonal projection operator onto $\calA(\productspace{N}{K})$, with a slight abuse of notation, is also denoted by $\scrA$: For any $N$-order tensor $\bX$,
$$\scrA[\bX]  (k_1,\ldots,k_N):=\dfrac{1}{N!}\sum_{\pi \in \calS_N}\sgn(\pi)\bX (k_{\pi(1)},\dots,k_{\pi(N)}).$$
If a tensor $\bX$ is of a tensor product form $\bX = \sum_{i=1}^p \bx_{i1}\otimes \cdots \otimes \bx_{iN}$, then its projection onto $\calA(\bigotimes^N\C^K)$ can be directly calculated by
$$\scrA[\bX] = \dfrac{1}{N!}\sum_{\pi\in \calS_N}\sum_{i=1}^{p}\sgn(\pi)\bx_{i\pi(1)}\otimes\cdots\otimes\bx_{i\pi(N)}.$$

\par The following proposition gathers some key properties of antisymmetric tensors.
\begin{prop}[{\cite[Section 3.5.1]{hackbusch2012tensor}}]\label{prop:properties of antisym tensors}
	Let $\bX \in \calA(\productspace{N}{K})$. Then:
	\begin{itemize}
		\item If there exist $i\neq j$ such that $k_i=k_j$, then $\bX(k_1,\ldots,k_N)= 0$.
		\item If $\calA(\productspace{N}{K}) \neq \{\mathbf{0}\}$, then $K\ge N$.
		\item If $K \ge N$, then $\dim \calA(\productspace{N}{K}) =\displaystyle\binom{K}{N}$.
	\end{itemize}
\end{prop}

\par To write out the basis of $\calA(\productspace{N}{K})$, we introduce the multi-index set
\begin{equation}
	\label{eq:index_set}
	\Lambda := \lrbrace{\bk=(k_1,\ldots,k_N): 1\le k_1<\cdots < k_N\le K}.
\end{equation}
Notice that $\abs{\Lambda}=\binom{K}{N}$, and the entries of an antisymmetric tensor $\bX$ indexed by the multi-indices in $\Lambda$ are independent. For each $\bk=(k_1,\ldots,k_N)\in \Lambda$, we define the corresponding basis tensor $\bE_{\bk} \in \calA(\productspace{N}{K})$ as
\begin{equation}
	\label{eq:basis_tensor_product}
	\bE_{\bk} := \sum_{\pi \in \calS_N} \sgn(\pi) \be_{k_{\pi(1)}} \otimes \cdots \otimes \be_{k_{\pi(N)}},
\end{equation}
where $\be_{j}\in\C^K$ is the $j$-th standard unit vector in $\C^K$. By the definition of $\bE_{\bk}$, we have
\begin{equation}
	\label{eq:basis_anti_per}
	\bE_{\bk}(\ell_1,\dots,\ell_N) = \begin{cases}
		1 & \text{if }\{\ell_1,\dots,\ell_N\} \text{ is an even permutation of } \{k_1,\dots,k_N\},\\
		-1 & \text{if }\{\ell_1,\dots,\ell_N\} \text{ is an odd permutation of } \{k_1,\dots,k_N\},\\
		0 & \text{otherwise}.
	\end{cases}
\end{equation}
Let $c_{\bk} := \bX (k_1,\ldots,k_N)$ for any $\bk\in\Lambda$. The expansion of $\bX \in \calA(\productspace{N}{K})$ in the basis $\{\bE_{\bk}\}_{\bk\in\Lambda}$ is
\begin{equation}
	\label{eq:anti_tensor_expand}
	\bX = \sum_{\bk \in \Lambda} c_{\bk} \bE_{\bk}.
\end{equation}
When $K=N$, the multi-index set $\Lambda$ contains only one element, and thus there is only one basis tensor, denoted as
\begin{equation}
	\label{eq:det_tensor}
	\bE := \sum_{\pi \in \calS_N} \sgn(\pi) \be_{\pi(1)} \otimes \cdots \otimes \be_{\pi(N)}.
\end{equation}
This tensor is also known as the determinant tensor \cite{derksen2016nuclear}.


\section{Main results}
\label{sec_result}

\par We state and rigorously prove our main theoretical results in this section. In particular, we focus on the finite-dimensional $\bigotimes^N\calB(\Omega)$ and establish an exponential lower bound on the TPF rank of any nonzero antisymmetric function therein. This result is further applied to antisymmetric TNNs.

\subsection{Representation Complexity of Antisymmetric Tensor Product Functions in Finite-Dimensional Spaces}

\par We consider the situation $\calB(\Omega)=\calF_K(\Omega)$, where $\calF_K(\Omega)$ is a $K$-dimensional space spanned by the linearly independent basis functions $\lrbrace{\bfun_k}_{k=1}^{K}$ defined on $\Omega$, i.e.,
$$\calF_K(\Omega):=\text{span}\lrbrace{\bfun_1,\dots,\bfun_K}.$$
\begin{remark}
	It should be remarked that the finite-dimensional setting described above aligns with common practice. Typically, the functions $\{\psi_{ij}\}$ in TPFs are first parameterized using a set of basis functions, followed by the implementation of a finite-dimensional discretized version \cite{bachmayr2023low}. Moreover, the results established in this finite-dimensional setting are also applicable to the TNNs parameterized by neural networks, as we will demonstrate in the next subsection.
\end{remark}

Before giving the main results, we provide the following lemma on the linear independence of TPFs.
\begin{lem}
	\label{lemma:linear_indep}
	For any $1\le k_1,\ldots,k_N\le K$, let the function $\Phi_{N,\bk}$ ($\bk:=(k_1,\ldots,k_N)$) be defined as $\Phi_{N,\bk}:=\bigotimes_{j=1}^N\bfun_{k_j}$. Then $\{\Phi_{N,\bk}\}_{\bk}$ are linearly independent.
\end{lem}
\begin{proof}
	We prove by mathematical induction on $N$. For the case of $N=1$, the conclusion holds trivially since $\lrbrace{\bfun_k}_{k=1}^K$ are already linearly independent. Now assume that the conclusion holds for the case of $N=m-1$, that is, the functions $\{\Phi_{m-1,\bk}\}_{\bk}$ are linearly independent. Consider the case of $N=m$.

	Suppose that there exist constants $\{c_{k_1,\dots,k_m}\}$ such that
	$$\sum_{k_1=1}^K\cdots \sum_{k_m=1}^K c_{k_1,\dots,k_m}\bigotimes_{j=1}^m\bfun_{k_j}=0. $$
	By rearranging terms, we obtain
	$$\sum_{k_1=1}^K \cdots \sum_{k_{m-1}=1}^K \lrbracket{\bigotimes_{j=1}^{m-1}\bfun_{k_j}}\otimes\lrbracket{\sum_{k_m=1}^K c_{k_1,\dots,k_m}\bfun_{k_m}}=0.$$
	Therefore, for any $\bx=(\bx_1,\dots,\bx_m)\in \Omega^m$, we have
		$$\sum_{k_1=1}^K \cdots \sum_{k_{m-1}=1}^K \lrbracket{ \sum_{k_m=1}^K c_{k_1,\dots,k_m}\varphi_{k_m}(\bx_m)}\prod_{j=1}^{m-1}\varphi_{k_j}(\bx_j) = 0. $$
	By induction, we conclude that
	$$\sum_{k_m=1}^K c_{k_1,\dots,k_m}\bfun_{k_m}(\bx_m)=0, \quad\forall\  1\le k_1,\dots, k_{m-1}\le K,~\bx_m\in\Omega.$$
	Since $\lrbrace{\bfun_k}_{k=1}^K$ are linearly independent, it follows that $c_{k_1,\dots,k_m}=0$ for any $1\le k_1,\ldots,k_m\le K$, which implies that the conclusion also holds for the case of $N=m$. The proof is complete.
\end{proof}

\par For any $f\in\bigotimes^N\calF_K(\Omega)$, we could identify its TPF rank as the CP rank of an $N$-order tensor, as established by the following theorem.

\begin{The}
	\label{the:general_main_result}
	For any $f \in \bigotimes^N\calF_K(\Omega)$, there exists a tensor $\bX \in \productspace{N}{K}$ such that
	$$\rank_{\calF_K(\Omega)}(f) = \rank_{\cp}(\bX).$$
\end{The}
\begin{proof}
	Let $p:=\rank_{\calF_K(\Omega)}(f)$. Then there exist functions $\{\psi_{ij}\}\subseteq\calF_K(\Omega)$ such that $f$ can be written in the form $f=\sum_{i=1}^p\bigotimes_{j=1}^N\psi_{ij}$. Expanding each $\psi_{ij}$ in terms of the basis functions in $\calF_K(\Omega)$, we obtain
	$$\psi_{ij}=\sum_{k=1}^K c_{ijk}\bfun_k,$$
	where $\{c_{ijk}\}\subseteq\C$ are the expansion coefficients. Thus, the TPF $f$ satisfies
	\begin{align}
		f(\rr) &= \sum_{i=1}^p \prod_{j=1}^N\sum_{k=1}^K c_{ijk}\bfun_k(\rr_j) \nonumber\\
		&=\sum_{i=1}^p\sum_{1\le k_1,\dots,k_N\le K}\prod_{j=1}^N c_{ijk_j}\bfun_{k_j}(\rr_j)\nonumber\\
		&=\sum_{1\le k_1,\dots,k_N\le K}\lrbracket{\sum_{i=1}^p\prod_{j=1}^Nc_{ijk_j}}\bfun_{k_1}(\rr_1)\cdots\bfun_{k_N}(\rr_N)\label{eq:TPF finite expansion}
	\end{align}
	for any $\rr\in\Omega^N$.
	
	Define a tensor $\bX\in\productspace{N}{K}$ by $\bX(k_1,\dots,k_N):=\sum_{i=1}^p \prod_{j=1}^N c_{ijk_j}$, and let $\bx_{ij}=\lrbracket{c_{ij1},\dots,c_{ijK}}^{\top}\in\C^K$ for $1\le i\le p$ and $1\le j\le N$. Then
	\begin{equation}
		\label{eq:tensor_cp_form}
		\bX = \sum_{i=1}^p \bx_{i1}\otimes \cdots \otimes \bx_{iN},
	\end{equation}
	Therefore, the CP rank of $\bX$ is at most $p$, and it follows that
	$$\rank_{\calF_K(\Omega)}(f)=p \ge \rank_{\cp}(\bX).$$

	On the other hand, let $r:= \rank_{\cp}(\bX)$, and denote the CP decomposition of $\bX$ as
	$$\bX = \sum_{i=1}^r \tilde\bx_{i1}\otimes \cdots \otimes \tilde\bx_{iN},$$
	where $\tilde\bx_{ij}=\lrbracket{\tilde{c}_{ij1},\dots,\tilde{c}_{ijK}}^{\top}\in\C^K$ for $1\le i\le p$ and $1\le j\le N$. Similar to Eq. \eqref{eq:TPF finite expansion}, we have
	$$f(\rr) = \sum_{i=1}^r \prod_{j=1}^N\sum_{k=1}^K \tilde{c}_{ijk}\bfun_k(\rr_j).$$
	Let $\tilde{\psi}_{ij}:=\sum_{k=1}^K \tilde{c}_{ijk}\varphi_k$, we can write $f = \sum_{i=1}^r \bigotimes_{j=1}^N \tilde{\psi}_{ij}$, implying that
	$$r = \rank_{\cp}(\bX) \ge \rank_{\calF_K(\Omega)}(f),$$
	which completes the proof.
\end{proof}

\par We then concentrate on antisymmetric TPFs.

\begin{cor}
	\label{cor:main_result}
	Assume that $K \ge N$ and $f\in\calA(\bigotimes^N\calF_K(\Omega))\backslash\lrbrace{0}$. Then the TPF rank of $f$ satisfies
	$$\rank_{\calF_K(\Omega)}(f) \ge \min_{\bX \in \calA(\productspace{N}{K})\backslash \lrbrace{\boldsymbol{0}}}\rank_{\cp}(\bX).$$
\end{cor}
\begin{proof}
	Assume $f$ has the expansion as in Eq. \eqref{eq:TPF finite expansion}, and define the corresponding tensor $\bX$ as in Eq. \eqref{eq:tensor_cp_form}. Applying the antisymmetrizer to $f$, we obtain
	\begin{align}
		\scrA[f](\rr) &= \dfrac{1}{N!}\sum_{\tau\in \calS_N}\sgn(\tau)\sum_{i=1}^p\prod_{j=1}^N\sum_{k=1}^Kc_{ijk}\bfun_k(\rr_{\tau(j)})\nonumber\\
		& (\text{Let } \pi = \tau^{-1}, \ell=\tau(j), \text{ and thus } j=\pi(\ell) )\nonumber\\
		&=\dfrac{1}{N!}\sum_{\pi\in \calS_N}\sgn(\pi)\sum_{i=1}^p\prod_{\ell=1}^N \sum_{k=1}^Kc_{i\pi(\ell)k}\bfun_k(\rr_\ell)\nonumber\\
		&=\dfrac{1}{N!}\sum_{\pi \in \calS_N}\sgn(\pi) \sum_{1\le k_1,\dots,k_N\le K}\lrbracket{\sum_{i=1}^p\prod_{j=1}^Nc_{i\pi(j)k_j}}\bfun_{k_1}(\rr_1)\cdots\bfun_{k_N}(\rr_N)\nonumber\\
		&= \sum_{1\le k_1,\dots,k_N\le K} \lrbracket{ \dfrac{1}{N!}\sum_{\pi \in \calS_N}\sgn(\pi) \lrbracket{\sum_{i=1}^p\prod_{j=1}^Nc_{i\pi(j)k_j}} }\bfun_{k_1}(\rr_1)\cdots\bfun_{k_N}(\rr_N)\nonumber\\
		& (\text{Let } \tau = \pi^{-1}, \ell=\pi(j), \text{ and thus } j=\tau(\ell) )\nonumber\\
		&= \sum_{1\le k_1,\dots,k_N\le K} \lrbracket{ \dfrac{1}{N!}\sum_{\tau \in \calS_N}\sgn(\tau) \lrbracket{\sum_{i=1}^p\prod_{\ell=1}^Nc_{i\ell k_{\tau(\ell)}}} }\bfun_{k_1}(\rr_1)\cdots\bfun_{k_N}(\rr_N) \nonumber\\
		&=\sum_{1\le k_1,\dots,k_N\le K} \lrbracket{ \dfrac{1}{N!}\sum_{\tau \in \calS_N}\sgn(\tau) \bX(k_{\tau(1)},\dots,k_{\tau(N)}) }\bfun_{k_1}(\rr_1)\cdots\bfun_{k_N}(\rr_N) \label{eq:Antisym TPF finite expansion}
	\end{align}
	for any $\rr\in\Omega^N$.
	Since $f$ is antisymmetric, we have $f = \scrA[f]$. Noticing that $\lrbrace{\bfun_k}_{k=1}^K$ are linearly independent, we conclude after comparing Eqs. \eqref{eq:TPF finite expansion} and \eqref{eq:Antisym TPF finite expansion} and using Lemma \ref{lemma:linear_indep} that the tensor $\bX$ satisfies
	\begin{equation}
		\label{eq:expand_relation}
		\begin{cases}
			\bX(k_{1},\dots,k_{N}) = \dfrac{1}{N!}\displaystyle\sum_{\tau\in \calS_N}\sgn(\tau)\bX(k_{\tau(1)},\dots,k_{\tau(N)}),\ \forall~1\le k_1,\dots,k_N \le K, \\
			\displaystyle\exists\ k_1,\dots,k_N,\ \text{s. t. }\bX(k_{1},\dots,k_{N}) \neq 0.
		\end{cases}
	\end{equation}
	Thus, Eq. \eqref{eq:expand_relation} implies that $\bX\in \calA(\productspace{N}{K})$ and $\bX\neq 0$. By Theorem \ref{the:general_main_result}, it follows that
	$$\rank_{\calF_K(\Omega)}(f)=\rank_{\cp}(\bX) \ge \min_{\bY \in \calA(\productspace{N}{K})\backslash \lrbrace{\boldsymbol{0}}}\rank_{\cp}(\bY),$$
	which completes the proof.
\end{proof}
\begin{remark}
	The condition $K\ge N$ is necessary to avoid the trivial case. If $K <N$, by the basic property of antisymmetric tensors (see Proposition \ref{prop:properties of antisym tensors}), it must hold that $f\equiv 0$.
\end{remark}
\par This corollary establishes that the TPF rank of any $f\in\calA(\bigotimes^N\calF_K(\Omega))\backslash\lrbrace{0}$ is not smaller than the lowest CP rank of nonzero tensors in $\calA(\bigotimes^N\C^K)$. To estimate this lowest CP rank, we turn to analyze the CP ranks of the basis tensors $\{\bE_{\bk}\}_{\bk\in\Lambda}$, as suggested by Eq. \eqref{eq:anti_tensor_expand}. For this purpose, we introduce an existing lower bound of $\rank_{\cp}(\bE)$, where $\bE$ is the determinant tensor in Eq. \eqref{eq:det_tensor}.
\begin{lem}[\cite{derksen2016nuclear}]
	\label{lemma:det_rank}
	Let $\bE\in \calA(\productspace{N}{N})$ be the determinant tensor defined in Eq. \eqref{eq:det_tensor}. Then
	$$ \displaystyle\binom{N}{\lfloor \frac{N}{2} \rfloor} \le \rank_{\cp} (\bE) \le 
	N!\cdot \lrbracket{\dfrac{5}{6}}^{\lfloor \frac{N}{3} \rfloor}.$$
\end{lem}

\par The antisymmetric basis tensors $\{\bE_{\bk}\}_{\bk\in\Lambda}$ defined in Eq. \eqref{eq:basis_tensor_product} generalize the determinant tensor $\bE$ to higher dimensions. We now demonstrate that their CP ranks are identical.
\begin{The}
	\label{the:rank_eq}
	Assume that $K\ge N$. Let $\Lambda$ be the multi-index set defined in Eq. \eqref{eq:index_set}. For any multi-index $\bk \in \Lambda$, 
	$$\rank_{\cp}(\bE_{\bk})=\rank_{\cp}(\bE).$$
\end{The}
\begin{proof}
	Let $\bk:=\lrbracket{k_1,\dots,k_N}$, $p:=\rank_{\cp}(\bE)$ and $p_{\bk}:=\rank_{\cp}(\bE_{\bk})$. By the definition of CP rank, there exist vectors $\ba_{ij}\in \C^N$, $j=1,\dots,N$, $i=1,\dots,p$, such that
	$$\bE = \sum_{i=1}^p \ba_{i1}\otimes \cdots \otimes \ba_{iN}.$$
	We extend each $\ba_{ij}$ to $\C^K$ by adding zero entries and denote the resulting vector by $\tilde{\ba}_{ij}\in\C^K$, which satisfies that the $\ell$-th entry in $\ba_{ij}$ is the $k_\ell$-th entry in $\tilde{\ba}_{ij}$. Based on $\{\tilde\ba_{ij}\}$, we define 
	\begin{equation}
		\bB_{\bk} := \sum_{i=1}^p \tilde{\ba}_{i1}\otimes \cdots \otimes \tilde{\ba}_{iN}\in\bigotimes^N\C^K.
		\label{eq:tmp tensor}
	\end{equation}
	
	\par In the following, we discuss the value of $\bB_{\bk}(\ell_1,\ldots,\ell_N)$ with $1\le\ell_1,\ldots,\ell_N\le K$ and will finally end up with $\bB_{\bk}=\bE_{\bk}$.\\[0.2cm]
	\noindent\textbf{Case I.} There exists an $\ell_j \not\in\lrbrace{k_1,\dots,k_N}$. By the definition of $\tilde{\ba}_{ij}$, we have $\bB_{\bk}(\ell_1,\dots,\ell_N)=0$.\\[0.2cm]
	\noindent\textbf{Case II.} For any $j\in\{1,\ldots,N\}$, $\ell_j\in\lrbrace{k_1,\dots,k_N}$.
	\begin{itemize}
		\item If there exist $j\ne m$ such that $\ell_j = \ell_m$, by the definition of $\tilde{\ba}_{ij}$ and $\bE$, $\bB_{\bk}(\ell_1,\dots,\ell_N)=0$.
		
		\item If $\ell_j\neq \ell_m$ for any $j\neq m$, then there exists a permutation $\pi\in\calS_N$ such that $\ell_j = k_{\pi(j)}$ for $j=1,\ldots,N$. We thus have
		$$\bB_{\bk}(\ell_1,\dots,\ell_N)=\bE(\pi(1),\dots,\pi(N)).$$ 
		By the definition of $\bE$, if $\pi$ is even, $\bB_{\bk}(\ell_1,\dots,\ell_N)=1$; otherwise, $\bB_{\bk}(\ell_1,\dots,\ell_N)=-1$.
	\end{itemize}
	By the above two cases and the definition of $\bE_{\bk}$ in Eq. \eqref{eq:basis_anti_per}, we have $\bE_{\bk}=\bB_{\bk}$. From Eq. \eqref{eq:tmp tensor} and the definition of CP rank, $p_{\bk}\le p$.
	
	\par Conversely, there exist vectors $\tilde{\bb}_{ij}\in \C^K$, $j=1,\dots,N$, $i=1,\dots,p$, such that
	$$\bE_{\bk} = \sum_{i=1}^{p_{\bk}} \tilde{\bb}_{i1}\otimes \cdots \otimes \tilde{\bb}_{iN}.$$
	Let $\bb_{ij}\in\C^N$ be the restriction of $\tilde{\bb}_{ij}$ to the entries indexed by $\bk$, that is, the $\ell$-th entry of $\bb_{ij}$ is the $k_{\ell}$-th entry of $\tilde{\bb}_{ij}$ ($\ell=1,\ldots,N$). Similar arguments then yield
	$$\bE = \sum_{i=1}^{p_{\bk}} \bb_{i1}\otimes \cdots \otimes\bb_{iN},$$
	implying $p \le p_{\bk}$. Therefore, $p = p_{\bk}$ as desired.
\end{proof}

\par With the above theorem, we obtain the following inequalities for the CP rank of any nonzero antisymmetric tensor.
\begin{cor}
	\label{cor:bound_rank}
	Assume that $K\ge N$. Let $\bX \in \calA\lrbracket{\productspace{N}{K}}\backslash\lrbrace{\boldsymbol{0}}$. Then
	$$\displaystyle\binom{N}{\lfloor \frac{N}{2} \rfloor} \le \rank_{\cp}(\bX)\le N!\cdot\displaystyle\binom{K}{N}\cdot \lrbracket{\dfrac{5}{6}}^{\lfloor \frac{N}{3} \rfloor}.$$
\end{cor}
\begin{proof}
	By using the expansion of the antisymmetric tensor $\bX$ in Eq. \eqref{eq:anti_tensor_expand}, the definition of CP rank, Lemma \ref{lemma:det_rank} and Theorem \ref{the:rank_eq}, we obtain
	$$\rank_{\cp}(\bX)\le\abs{\Lambda}\cdot\rank_{\cp}(\bE_{\bk})\le N!\cdot\displaystyle\binom{K}{N}\cdot \lrbracket{\dfrac{5}{6}}^{\lfloor \frac{N}{3} \rfloor}.$$
	
	\par Now consider the other side. Choose $\boldsymbol{u}:=\lrbracket{k_1,\dots,k_N}\in \Lambda$ that satisfies $c_{\boldsymbol{u}}:=\bX(k_1,\dots,k_N)\neq 0$. Define the truncated tensor $\tilde{\bX}$ of $\bX$ by
		$$\tilde{\bX}(\ell_1,\dots,\ell_N):=\begin{cases}
			\bX(\ell_1,\dots,\ell_N), & \text{ if } \ell_i\in \lrbrace{k_1,\dots,k_N},~\text{for}~1\le i\le N,\\
			0, & \text{ otherwise.}
		\end{cases}$$
	It is straightforward to verify from Eq. \eqref{eq:basis_anti_per} that
	$$\supp(\tilde\bX)=\supp(\bE_{\bu})\quad\text{and}\quad\supp(\bE_{\bu'})\cap\supp(\bE_{\bu''})=\emptyset,\quad\forall~\bu',\bu''\in\Lambda, \bu'\neq \bu''.$$
	Therefore, from the expansion in Eq. \eqref{eq:anti_tensor_expand}, we see that $\tilde{\bX} = c_{\boldsymbol{u}}\bE_{\boldsymbol{u}}$. Let $p:=\rank_{\cp}(\bX)$. The CP decomposition of $\bX$ then takes the form
	$$\bX=\sum_{i=1}^p \bx_{i1}\otimes \cdots \otimes \bx_{iN},$$
	where $\bx_{ij}\in \C^K$ for $i=1,\ldots,p$ and $j=1,\ldots,N$. By using similar arguments as in the proof of Theorem \ref{the:rank_eq}, we have that
	$$\tilde{\bX}=\sum_{i=1}^p \tilde{\bx}_{i1}\otimes \cdots \otimes \tilde{\bx}_{iN},$$
	where $\tilde{\bx}_{ij}\in \C^K$ shares the $(k_1,\dots,k_N)$-th entry with \(\bx_{ij}\) while letting others be zero.
	By the definition of CP rank, Lemma \ref{lemma:det_rank}, and Theorem \ref{the:rank_eq}, it follows that
	$$p=\rank_{\cp}(\bX)\ge \rank_{\cp}(\tilde{\bX})=\rank_{\cp}(\bE_{\boldsymbol{u}})\ge \displaystyle\binom{N}{\lfloor \frac{N}{2} \rfloor}.$$
	The proof is complete.
\end{proof}

\par Notice that the minimum CP rank of any antisymmetric tensor depends only on $N$. Thus, after taking the minimum over $\calA(\productspace{N}{K})\backslash\lrbrace{\boldsymbol{0}}$, the lower bound remains unchanged. Combining Corollary \ref{cor:main_result} and Corollary \ref{cor:bound_rank} and applying the Stirling approximation $N!=\Theta(\sqrt{2\pi N}(N/e)^N)$, we obtain the following result.
\begin{cor}
	\label{cor:lower_bound_p}
	Assume that $K \ge N$ and $f\in\calA(\bigotimes^N\calF_K(\Omega))\backslash\lrbrace{0}$. Then the TPF rank of $f$ satisfies
	$$\rank_{\calF_K(\Omega)}(f) \ge \Theta\lrbracket{\dfrac{2^N}{\sqrt{N}}}.$$
\end{cor}

Corollary \ref{cor:lower_bound_p} demonstrates that the TPF rank of any antisymmetric TPF in a finite-dimensional space grows at least exponentially with the dimension. Moreover, this exponential lower bound is independent of the basis size $K\ge N$. This independence originates from the intrinsic structure of the determinant tensor in Eq.~\eqref{eq:det_tensor}. For $K<N$, Proposition \ref{prop:properties of antisym tensors} implies that the only antisymmetric TPF is the trivial one $f\equiv0$.
\begin{remark}
	In a recent work \cite{huang2023backflow}, the authors take the so-called backflow ans\"atze--a generalized Slater determinant--as the basis function and show that, in the worst case, approximating a nonzero antisymmetric polynomial may require exponentially many determinants. In comparison, our work takes the TPF as the basis function and adopts a best-case perspective: We establish that even the \textit{minimum} number of terms in representing a nonzero antisymmetric function grows exponentially with the problem dimension.
\end{remark}

\subsection{Representation Complexity of Antisymmetric Tensor Neural Networks}

\par Consider the TNN defined in Eq. \eqref{eq:tnn_wf}. We begin with the cases where each subnetwork has only one hidden layer with a width of $m\in\N$. In this case, each univariate function takes the form
\begin{equation}
	\label{eq:tnn_psi}
	\psi_{ijk}(r_{jk};\btheta_{jk})=\sum_{\ell=1}^m a_{ijk\ell}\cdot\sigma(\omega_{jk\ell}r_{jk}+b_{jk\ell}) +c_{ijk},
\end{equation}
where $\btheta_{jk}$ collects $\{a_{ijk\ell}\}_{i\ell}$, $\{c_{ijk}\}_i$, $\{\omega_{jk\ell}\}_\ell$, $\{b_{jk\ell}\}_\ell\subseteq\C$, and $\sigma(\cdot)$ is an activation function. We aim to analyze the TPF rank of antisymmetric TNNs.

\par To apply the theoretical results in the previous subsection, we proceed to formulate the TNNs into TPFs in finite-dimensional spaces.
To this end, for any $i\in\{1,\ldots,p\}$ and $j\in\{1,\ldots,N\}$, define 
$$\psi_{ij}(\rr_j;\btheta_j):=\prod_{k=1}^d\psi_{ijk}(r_{jk};\btheta_{jk}),\quad\forall~\rr_j=\lrbracket{r_{j1},\dots,r_{jd}}^{\top}\in\Omega,$$
where $\btheta_j:=\lrbracket{\btheta_{j1},\dots,\btheta_{jd}}^{\top}$. Recalling the definition of $\psi_{ijk}$ in Eq. \eqref{eq:tnn_psi}, one can thus verify that $\psi_{ij}(\cdot;\btheta_j)\in\tilde\calF_{N,m,1}(\Omega):=\mathrm{span}\lrbrace{\tilde\bfun_{t,\calI,\bell}}$, where
$$\tilde\bfun_{t,\calI,\bell}(\hat\rr):=\prod_{k\in\calI}\sigma(\omega_{tk\ell_k}\hat r_k+b_{tk\ell_k}),\quad\forall~\hat\rr=\lrbracket{\hat r_{1},\dots,\hat r_{d}}^{\top}\in\Omega,$$
for $t\in\{1,\ldots,N\}$, $\calI\subseteq\{1,\ldots,d\}$, $\bell=(\ell_k)_{k\in\calI}$ with each $\ell_k\in\{1,\ldots,m\}$. Note that when $\calI=\emptyset$, the function is defined to be identically equal to 1.
Consequently, the TNN $f_{\btheta}\in\bigotimes^N \tilde\calF_{N,m,1}(\Omega)$. Although the functions $\{\tilde\bfun_{t,\calI,\bell}\}$ are not necessarily linearly independent, the space $\tilde\calF_{N,m,1}(\Omega)$ is finite-dimensional for fixed parameters.

\par Using Corollary \ref{cor:lower_bound_p}, we obtain the following bound for any antisymmetric TNN with one hidden layer and a width of $m$.
\begin{cor}
	\label{cor:one_hidden_tnn}
	For an non-zero antisymmetric TNN with one hidden layer, a width of $m$, and any activation function, for any network parameters, its TPF rank $p$ satisfies
	$$p \ge \Theta(2^N/\sqrt{N}).$$
\end{cor}
\begin{proof}
	From the discussions above, the TNN $f_{\btheta}$ with any fixed set of network parameters always belongs to the finite-dimensional function space $\tilde{\calF}_{N,m,1}(\Omega)$. By Corollary \ref{cor:lower_bound_p}, for any fixed network parameters, we have
		$$p\ge \Theta(2^N/\sqrt{N}).$$
		Since varying the parameters only alters the basis functions without changing the intrinsic finite dimensionality, this lower bound holds uniformly for all network parameters.
\end{proof}

\par Notice that the above result is applicable under any fixed network architectures. The only modification required is to adjust the finite-dimensional space. 

\begin{cor}
	\label{cor:tnn_complex}
	For an non-zero antisymmetric TNN with any fixed finite depths and widths and any activation function, for any network parameters, its TPF rank $p$ satisfies
	$$p \ge \Theta(2^N/\sqrt{N}).$$
\end{cor}

\begin{remark}
	During the training process, the TNN transfers among different finite-dimensional function spaces. Corollary \ref{cor:tnn_complex} states that only when $p \ge \Theta(2^N/\sqrt{N})$ can it possibly reach a space that contains a nonzero antisymmetric function.
\end{remark}

\begin{remark}
	According to our theoretical results, for a three-electron system (i.e., $N=3$), the TPF rank of the electronic wave function in finite-dimensional spaces is at least $\binom{3}{1}=3$. On the other hand, the study \cite{ilten2016product} demonstrates that the CP rank of the third-order determinant tensor is exactly five. These indicate that a TPF rank of $p=5$ is sufficient to represent antisymmetry in finite-dimensional spaces. Notably, this does not contradict the experimental results in \cite{liao2024solvingschrodingerequationusing}, where a rank of $p=50$ was employed in TNNs. Our results establish a theoretical condition required to ensure the existence of nonzero antisymmetric solutions in finite-dimensional spaces, whereas the TNNs in \cite{liao2024solvingschrodingerequationusing} are specific members within the associated spaces. 
\end{remark}


\section{Conclusion and Discussion}
\label{sec_discussion}

\par Recent works reveal that the TPF approximations (specifically, the TNNs) can incur extraordinarily high computational costs to achieve desirable accuracy on quantum many-body problems, even for systems with as few as three electrons. This stands in sharp contrast to their success on other high-dimensional applications. Preliminary numerical results point to their representation complexity in representing antisymmetry, a fundamental property that the electronic wave function should always satisfy. In this work, we rigorously establish an exponential lower bound on the TPF rank for the antisymmetric TPFs in finite-dimensional spaces. In other words, the minimum number of involved terms for the TPFs in finite-dimensional spaces to ensure antisymmetry increases exponentially with the problem dimension. Notably, this finite-dimensional setting is compatible with the traditional discretization methods and also includes the TNNs (the TPFs parameterized by neural networks) as special cases. In our proof, we mainly leverage the connection between the antisymmetric TPFs and antisymmetric tensors and dig into the CP rank of the latter.

\par Our theoretical results offer new insights and perspectives for applications with antisymmetry requirements. Firstly, low-rank TPFs, while favorable due to the separability of integrals, cannot be antisymmetric in high-dimensional settings. For example, in the case of $N=20$, representing antisymmetry requires a rank of at least $1.8 \times 10^5$, making such representations impractical. However, this does not necessarily imply an unacceptably high computational cost in implementations especially when certain algebraic structures are available. In contrast, our findings partly explain why determinant-based constructions have become the standard choice in modern practice. 
Alternatively, one could resort to approximations based on other tensor formats, such as tensor train formats \cite{oseledets2011tensor}, whose capabilities to represent antisymmetry have not been studied. 
Secondly, quantitatively investigating the impact of errors in representing antisymmetry on relevant applications would be valuable. As an analogy, the principal component analysis in statistics essentially approximates a covariance matrix with a low-rank one, while retaining most of the critical information. Understanding the trade-off between computational efficiency and accuracy in antisymmetry representation could guide further developments.

    \paragraph{Acknowledgements.} The work of Yukuan Hu has received funding from the European Research Council (ERC) under the European Union’s Horizon 2020 Research and Innovation Programme (Grant Agreement EMC2 No. 810367). The work of Xin Liu was supported in part by the National Natural Science Foundation of China (12125108, 12021001, 12288201), and RGC grant JLFS/P-501/24 for the CAS AMSS–PolyU Joint Laboratory in Applied Mathematics.

\begin{appendices}
	\section{Experimental Settings}
	\label{sec_appendix}

	\par For the one-dimensional system, we consider the Hamiltonian with soft-Coulomb interaction \cite{zhou2024multilevel}:
	$$\scrH:=-\dfrac{1}{2}\sum_{i=1}^N\Delta_{r_i}-\sum_{i=1}^N\sum_{I=1}^M\dfrac{Z_I}{\sqrt{1+(r_i-R_I)^2}}+\sum_{1\le i<j \le N}\dfrac{1}{\sqrt{1+(r_i-r_j)^2}},$$
	where $\lrbrace{r_i}_{i=1}^N$ and $\lrbrace{R_I}_{I=1}^M$ represent electron and nucleus positions respectively, and $Z_I$ is the charge of the $I$-th nucleus. 
	
	\par We consider a spinless system. For the non-antisymmetric TNN wave function $f_{\btheta}$ defined in Eq. \eqref{eq:TNN_general_fun}, the loss function is chosen as in \cite{liao2024solvingschrodingerequationusing}:
	$$L^{\text{TNN}}(\btheta):=\dfrac{\inner{f_{\btheta}, \scrH f_{\btheta}}}{\inner{f_{\btheta},f_{\btheta}}} + \beta \sum_{1\le i<j\le N}\dfrac{\inner{f_{\btheta}, f_{\btheta}\circ \scrT_{ij}}}{\inner{f_{\btheta}, f_{\btheta}}},$$
	where $\inner{\cdot,\cdot}$ is the inner product in $\calL^2$, and $\beta >0$ is the penalty parameter.
	
	\par For the antisymmetrized TNN wave function, we apply the antisymmetrizer $\scrA$ defined in Eq. \eqref{eq:anti_proj} to $f_{\btheta}$, yielding $f_{\btheta}^{\text{anti}}$:
	$$f_{\btheta}^{\text{anti}}(r_1,\dots,r_N):=\dfrac{1}{N!}\sum_{\pi\in \calS_N}\sum_{i=1}^p \text{sgn}(\pi)\prod_{j=1}^N \psi_{ij}(r_{\pi(j)};\btheta_j).$$
	The corresponding loss function is:
	$$L^{\text{anti}}(\btheta):=\dfrac{\inner{f_{\btheta}^{\text{anti}}, \scrH f_{\btheta}^{\text{anti}}}}{\inner{f_{\btheta}^{\text{anti}}, f_{\btheta}^{\text{anti}}}}.$$
	
	\par Our experiments are conducted on one-dimensional Lithium ($\mathrm{Li}$) and Helium hydride ion ($\mathrm{HeH}^+$). For $\mathrm{Li}$, the nucleus is positioned at 0.0. For $\mathrm{HeH}^+$, the nuclei are located at 0.0 ($\mathrm{He}$) and 1.463 ($\mathrm{H}^+$). Since the loss functions only involve one- and two-dimensional integrals, we compute them using Gauss–Legendre quadrature. The integration interval is truncated to $\lrsquare{-10,10}$, divided uniformly into 30 subintervals, with 30 quadrature points applied in each subinterval.
	
	\par The neural network architecture and optimization settings are detailed in Table \ref{tab:parameters}, closely following \cite{liao2024solvingschrodingerequationusing}. All experiments were implemented in JAX \cite{jax2018github} on an Ubuntu 20.04.1 operating system and trained on GPUs (8× NVIDIA GeForce RTX 4090) with CUDA 12.0. Results for $p=4$ are shown in Figure \ref{fig:tnn_comparison}.
	\begin{table}[H]
		\centering
		\renewcommand{\arraystretch}{1.2} 
		\begin{tabular}{|c|c|c|}
			\hline
			\multicolumn{2}{|c|}{\textbf{Parameter}} & \textbf{Value} \\ \hline
			\multirow{4}{*}{\textbf{Neural Network}} 
			& Hidden layers & $L=2$  \\ \cline{2-3}
			& Hidden dimension & $m=20$  \\ \cline{2-3}
			& Activation function & $\sigma(x)=\tanh(x)$  \\ \cline{2-3}
			& Initialization & Default in JAX \\ \hline
			\multirow{5}{*}{\textbf{Optimization}} 
			& Optimizer & Adam  \\ \cline{2-3}
			& Initial learning rate & \(\eta_0=1.00 \times 10^{-3}\)  \\ \cline{2-3}
			& Learning rate decay & $\eta_k=0.7^{\lfloor k/3000 \rfloor}\eta_0$
			\\ \cline{2-3} 
			& Iterations & 50000 \\ \cline{2-3}
			& Penalty parameter & $\beta=200$ \\ \hline
		\end{tabular}%
		\caption{Neural network and optimization parameters.}
		\label{tab:parameters}%
	\end{table}
	
	\par The numerical experiments aim to show that the lack of antisymmetry can substantially degrade accuracy and efficiency. Such deficiency cannot be resolved easily by tuning hyperparameters. This argument is supported by additional experiments using either more complicated neural networks ($L=4$, $m=40$) or different learning rate schedules ($\eta_k = \eta_0/(1+\alpha k)$, $\alpha = 1\times 10^{-3}$) (see Figure \ref{fig:additional_Li}). The obtained results are qualitatively consistent with those in Figure \ref{fig:tnn_comparison}.
	\begin{figure}[htb]
		\centering
		\begin{minipage}[b]{0.49\textwidth}
			\centering
			\includegraphics[width=\textwidth]{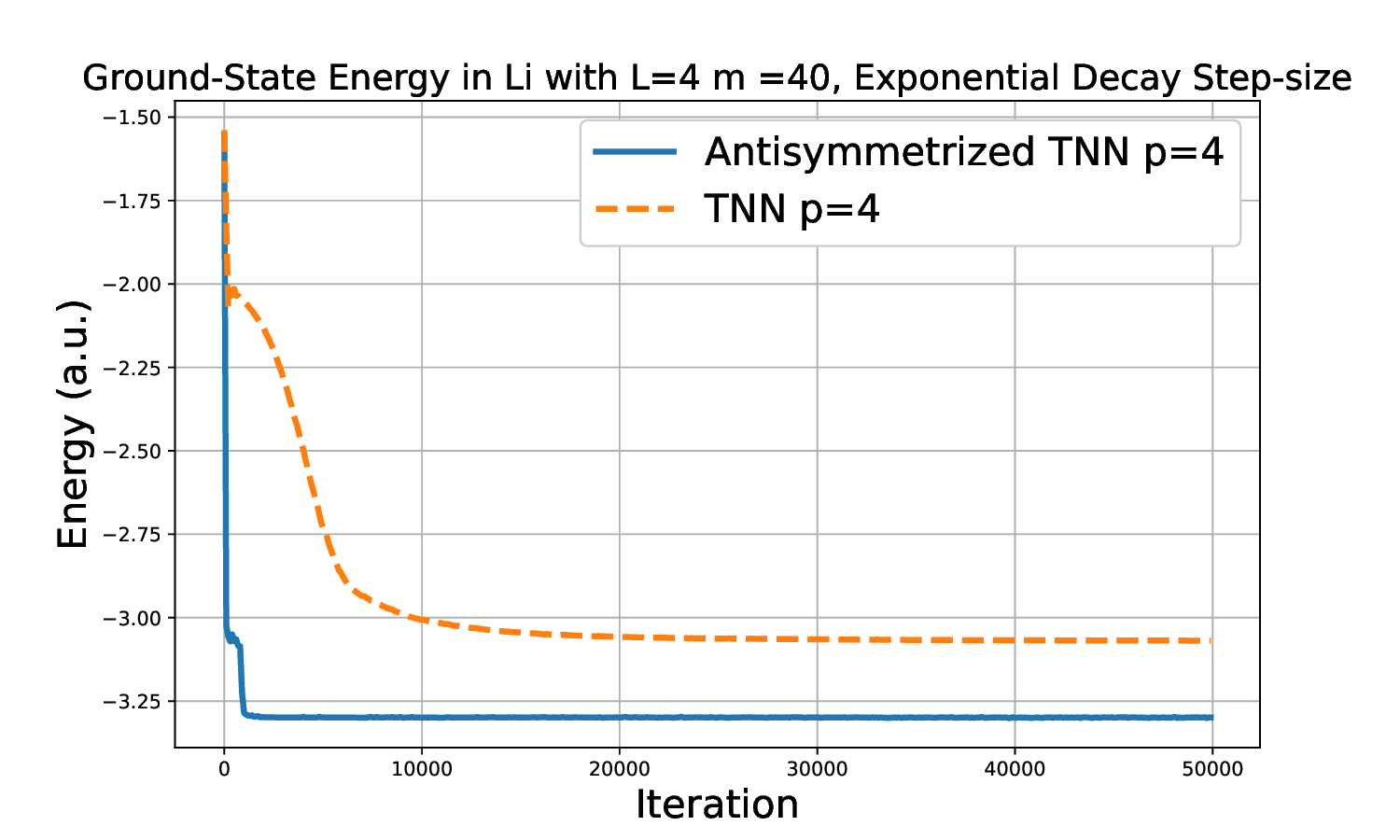}
		\end{minipage}
		\hfill
		\begin{minipage}[b]{0.49\textwidth}
			\centering
			\includegraphics[width=\textwidth]{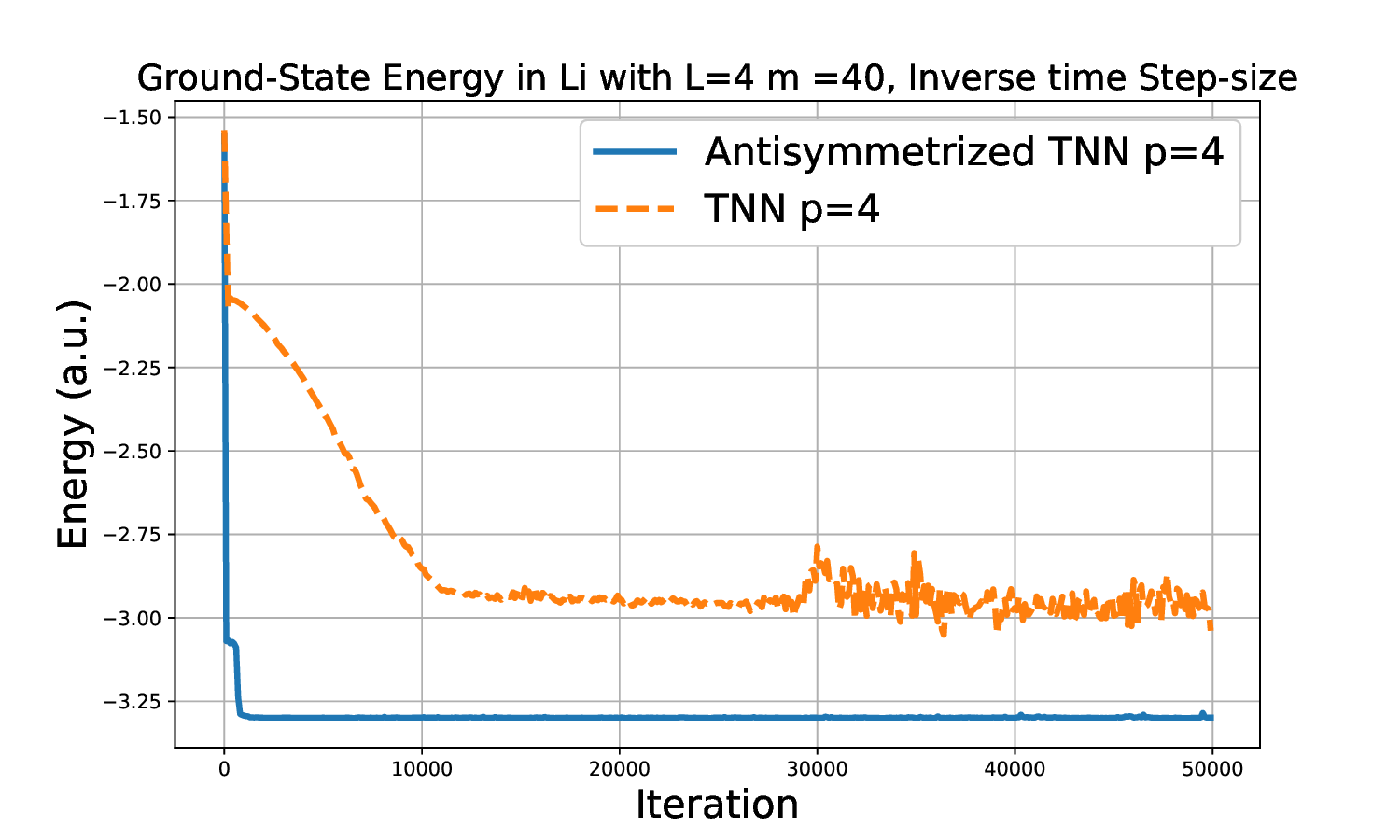}
		\end{minipage}
		\caption{Comparison of TNN approximations with and without antisymmetrization for the one-dimensional Li system using $L=4$ and $m=40$, under the exponential decay learning rate schedule (Left) and the inverse time schedule (Right).}
		\label{fig:additional_Li}
	\end{figure}

\end{appendices}

    \normalem
    \bibliographystyle{plain}
    \bibliography{ref_scm.bib}

\end{document}